\documentclass{article}
\usepackage[top=3cm, bottom=3cm, left=3cm , right=3cm]{geometry}
\usepackage[utf8]{inputenc}
\usepackage{authblk}

\usepackage{amsmath, amsthm} 
\usepackage{amssymb} 
\usepackage{amsfonts} 
\usepackage{mathrsfs} 
\usepackage{empheq} 
\usepackage{tikz}

\newtheorem{corollary}{Corollary}
\newtheorem{proposition}{Proposition}
\newtheorem{remark}{Remark}
\newtheorem{example}{Exemple}

\newcommand{\diag}{{\rm diag}}
\newcommand{\relu}{{\rm ReLU }}
\newcommand{\esup}{{\rm ess sup}}

\newcommand{\RR}{\mathbb{R}}

\newcommand{\Ccal}{\mathcal{C}}
\newcommand{\xcal}{\mathcal{X}}

\usepackage{comment}

\usepackage{cancel}
\usepackage{xcolor}
\newcommand{\so}[1]{{\color{magenta}#1}}

\usepackage{marginnote}

\title{MIQCQP reformulation of the ReLU neural networks Lipschitz constant estimation problem}
\author[1]{Mohammed Sbihi\thanks{mohammed.sbihi@recherche.enac.fr}}
\author[2]{Sophie Jan\thanks{sophie.jan@math.univ-toulouse.fr}}
\author[1,2]{Nicolas Couellan\thanks{nicolas.couellan@recherche.enac.fr}}
\affil[1]{Fédération ENAC ISAE-SUPAERO ONERA, Université de Toulouse, France}
\affil[2]{Institut de Math\'{e}matiques de Toulouse UMR 5219, Universit\'{e} de Toulouse; CNRS, UPS, F-31062 Toulouse Cedex 9, France.}

\date{}

\begin{document}

\maketitle
\begin{abstract}
It is well established that to ensure or certify the robustness of a neural network, its Lipschitz constant plays a prominent role. However, its calculation is NP-hard.
In this note, by taking into account activation regions at each layer as new constraints, we propose new quadratically constrained MIP formulations for the neural network Lipschitz estimation problem. The solutions of these problems give lower bounds and upper bounds of the Lipschitz constant and we detail conditions when they coincide with the exact Lipschitz constant. 
\end{abstract}

\section{Introduction}

Several studies have demonstrated the prominent role of the Lipschitz constant in the robustness of neural networks. It has been shown for example that it is related to generalization bounds of neural network classifiers \cite{tsuzuku2018, bartlett2017}. The Lipschitz constant expresses also the maximum variation of the neural network outputs and can therefore be used to derive robustness certificates when inputs are subject to random or adversarial perturbations \cite{Scaman2018LipschitzRO}. Furthermore, it has been used as a regularization term of the neural network training loss function to compute optimal neural network weights that achieve better robustness \cite{gouk2021}. Alternatively, constraints on the Lipschitz constant have been added in the training loss minimization problem to develop $1$-Lipschitz neural networks \cite{cisse2017}.         

The exact calculation of the Lipschitz constant of a neural network is a difficult problem. Even in the simple case of one hidden layer neural network, it can be shown that the problem is NP-hard~\cite{Scaman2018LipschitzRO}. Therefore, an estimation of the constant in the form of upper bounds (sometimes lower bounds) is usually sought. However, in the common case of ReLU networks, computing tight estimates of the constant is also an NP-hard problem (see Theorem 4 in \cite{Matt2020}). Several approaches have been proposed in the literature. They vary by the scope of the input being considered (global or local Lipschitz regularity), the order $p$ of $L_p$-Lipschitz regularity, or the underlying estimation method. 

Among these various approaches, Lipschitz certificates via semidefinite programming (SDP) are proposed in \cite{fazlyab2023}. By exploiting slope restriction properties of common activation functions, incremental quadratic constraints are formulated and the global $L_2$-Lipschitz constant problem is then expressed as a SDP. While it provides a nice convex formulation of the constant estimation, its application to real life neural networks architectures is limited by the computational complexity of available methods for solving SDP. Alternatively, in the case of ReLU networks, authors have considered mixed integer programming (MIP) approaches to derive exact or upper bounds of the Lipschitz constant. Indeed, in~\cite{Matt2020}, by showing that a ReLU network is a composition of MIP-encodable components and therefore itself MIP-encodable, the authors formulate the exact calculation of the local Lipschitz constant of ReLU networks as a MIP. In the worst case, MIP problems have exponential time complexity, however, in practice they are often solved in reasonable time.    

In this note, we propose new quadratically constrained MIP formulations for the neural network Lipschitz estimation problem. First, by taking into account activation regions at each layer as new constraints, we derive three new MIP formulations whose solutions give a lower bound $\underline{L}$ of the Lipschitz constant, a sequence $\{\underline{L}_\epsilon\}_{\epsilon>0}$ of lower bounds converging to $\underline{L}$ and an upper bound $\bar{L}$. We further show that $\underline{L}$ and $\bar{L}$ coincide and are equal to the true Lipschitz constant if the neural network is in general position as defined in \cite{Matt2020}. We also show that, except on a set of network parameters of Lebesgue measure $0$, $\underline{L}$ coincide with $\bar{L}$. Next, by reformulating the activation constraints as quadratic constraints, we propose equivalent Mixed Integer Quadratically Constrained Quadratic Program (MIQCQP) formulations. These new constrained problems have the benefit of reducing the search space in the branching phase involved in the MIP solving process. Furthermore, the specific quadratic structure of the objective and the constraints can also be exploited in the bounding phase of MIP solvers using quadratic convex relaxations and linearizations strategies as explained in \cite{Elloumi2019}. However, the study of the numerical solutions of these problems goes beyond the scope of this note that was only intended to explain the derivation of the MIQCQP formulation of the neural network Lipschitz estimation problem.

The note is organized in three sections. Section~\ref{prob} introduces the general problem of calculating the Lipschitz constant $L(f,\mathcal{X})$ of a neural network $f$ over an input set $\mathcal{X}$. Section~\ref{bounds} details the derivation of lower and upper bounds for $L(f,\mathcal{X})$. Finally, Section~\ref{MIQCQP} provides Mixed Integer Quadratically Constrained Quadratic Program reformulations for the problem of estimating  $L(f,\mathcal{X})$ using the bounds obtained in Section~\ref{bounds}.

\section{Problem statement}\label{prob}

We consider  ReLU  Multi-Layer-Perceptron function $f:\mathbb{R}^{n_0} \rightarrow \RR^{n_L}$, that is a composition of affine operators and element-wise \relu nonlinearities. More precisely, it may be encoded by:
$$
f(x)=T_L\circ \rho_{L-1}\circ T_{L-1}\circ \dots \circ \rho_1 \circ T_1 (x)
$$
where $T_k: \mathbb{R}^{n_{k-1}} \ni x \mapsto M_k x+ b_k\in \mathbb{R}^{n_k}$  is an affine function and $\rho_k : \mathbb{R}^{n_k} \rightarrow \mathbb{R}^{n_k}$ is the ReLU operator applied element-wise.
We denote by $\theta_k= T_k\circ \rho_{k-1}\circ T_{k-1} \circ \dots \circ \rho_1 \circ T_1$ the pre-activation output of the $k$-th layer and by $\theta_k^i$ the $i$-th component of $\theta_k$ (corresponding to the pre-activation of the $i$-th neuron of layer $k$). In the following, given an element $v\in \RR^n$, we denote its $i-th$ component by $v^i$ and we denote Hadamard product between two vectors $v_1$ and $v_2$  by $v_1 \odot v_2$.

We are interested in computing the quantity
$$\sup_{x,y\in \mathcal{X}}\frac{\|f(y)-f(x)\|}{\|y-x\|},$$
where  $\cal{X}$ is an open subset  of $\mathbb{R}^n$ and $\| . \|$ is a norm. When this quantity is finite,  we denote it by $L(f,\mathcal{X})$ and  we say that $f$ is locally Lipschitz over $\mathcal{X}.$ If $\mathcal{X}= \mathbb{R}^{n_0}$, then we denote the above quantity $L(f)$ and we simply say that $f$ is (globally) Lipschitz.

\section{Deriving lower and upper bounds of the Lipschitz constant of ReLU networks}\label{bounds}
We now derive an upper and lower bound\so{s} for $L(f,\xcal)$ by observing \cite[Theorem 1]{Matt2020} that 
\begin{equation}\label{eq:ess-sup}
  L(f,\mathcal{X})= \esup_{x\in \mathcal{X}}\sup_{G\in \mathcal{J}f(x)} \|G\|  
\end{equation} 
where $\mathcal{J}f(x)$ is the (Clarke) generalized jacobian of $f$ at $x$. 
Using recursively the Clarke jacobian Chain Rule (see \cite[Theorem 4]{Imbert2002}), we obtain the following bound 
\begin{equation}
    L(f,\mathcal{X})\leq \sup_{x\in \xcal} \left\{\|M_L\diag(g_{L-1}) M_{L-1}\cdots \diag(g_1) M_1\| \quad | \quad   g_k \in [0,1]^{n_k},\  g_k^i\in \partial \relu(\theta_k^i(x)) \right\}.
    \label{eq:ChainRule}
\end{equation}
Here $\partial \relu(x)$ is the subdifferential  of the \relu function:
$$\partial \relu (x)=\left\{
\begin{array}{ll}
\{0\} & \mbox{ if } x<0,\\
\mbox{[0,1]} & \mbox{ if } x=0,\\
\{1\} & \mbox{ if } x>0.
\end{array}
\right.
$$

An activation pattern for the \relu network $f$ is an assignment to each hidden neuron of   a sign 1 or 0:
$$(\sigma_1,\sigma_2,\dots,\sigma_{L-1}) \in \{0,1\}^{n_1} \times \{0,1\}^{n_2}\times \cdots \times \{0,1\}^{n_{L-1}}.$$
The activation region  in $\xcal$ corresponding to $(\sigma_1,\sigma_2,\dots,\sigma_{L-1})$ is
$$\mathcal{R}\left(\xcal;(\sigma_1,\sigma_2,\dots,\sigma_{L-1})\right)=\left\{x\in\xcal \quad |\quad
\left(\sigma_k^i-\frac{1}{2}\right) \theta_k^i(x)>0, \ \forall i\in\{1,\dots, n_k\},  \forall k\in \{1,\cdots,L-1\} \right\}.$$
When specifying an activation pattern, the signal assigned to a neuron $i$ from the $k-th$ layer determines whether it is on or off for inputs in the activation region  since the pre-activation of neuron is  positive (resp. negative) when $\sigma_k^i= 1$ (resp. $\sigma_k^i=0)$. 
Let us define $H_{k,i}=\{x\in \RR^{n_0} \ |\  \theta_k^i(x) =0\}$  (for $i\in\{1,\dots, n_k\}, k\in \{1,\cdots,L-1\}$) which can be thought of as “bent hyperplanes”. The non-empty activation regions is not else but the
connected components of $\xcal \setminus \cup_{k,i}H_{k,i}$ (Lemma 2 in \cite{Hanin2019}).  The Jacobian exists and is the same for all the points belonging to the same activation pattern. The Jacobian corresponding to a pattern $(\sigma_1,\cdots,\sigma_{L-1})$ is equal to $M_L\diag(\sigma_{L-1}) M_{L-1}\cdots \diag(\sigma_1) M_1$.
We obtain the following lower bound for $L(f,\xcal)$:
\begin{equation}
    L(f,\xcal)\geq \sup_{(\sigma_1,\sigma_2,\dots,\sigma_{L-1}) \in 
    \{0,1\}^{n_1+n_2+\cdots+n_{L-1}} |
\mathcal{R}(\xcal;\sigma)\neq \emptyset}
\|M_L\diag(\sigma_{L-1}) M_{L-1}\cdots \diag(\sigma_1) M_1\|.
\label{eq:borne_inf}
\end{equation}

Denoting $N(g) = \|M_L\diag(g_{L-1}) M_{L-1}\cdots \diag(g_1) M_1\|$, we are thus interested in solving the following two optimization problems to obtain upper and lower bounds of the Lipshitz constant of the ReLU neural network:
$$
(\mathcal{U}^*) \left\{\begin{array}{ll}
\max & N(g)\\
\mathrm{s.t. } \ & g_k^i \in \partial \relu(\theta_k^i(x)), \    \forall i\in\{1,\dots, n_k\}, k\in \{1,\cdots,L-1\}\\
& g_k \in [0,1]^{n_k}, \forall k\in \{1,\cdots,L-1\}\\
& x\in \xcal
\end{array}\right.
$$
and
$$
(\mathcal{L}^*) \left\{\begin{array}{ll}
\max & N(\sigma)\\
\mathrm{s.t. } \ & \mathcal{R}(\xcal;(\sigma_1,\cdots,\sigma_{L-1}))\neq\emptyset\\
& \sigma_k \in \{0,1\}^{n_k}, \forall k.
\end{array}\right.
$$

\begin{proposition}
For each $k\in \{1,\cdots,L-1\}$ and each $i\in\{1,\dots, n_k\}$,
the function $g_k^i \mapsto N(g)$ is convex.
\label{prop:per varable convexity}
\end{proposition}

\begin{proof}
The function is the composition of $g_k^i \mapsto M_L\diag(g_{L-1}) M_{L-1}\cdots \diag(g_1) M_1$ which is affine and the norm $\|.\|$ which is convex, therefore it is convex \cite[Proposition $2.1.5$]{JBHU_Lem1}. 
\end{proof}

\begin{proposition}
Problem~$(\mathcal{U}^*)$ is equivalent to the following one:
$$
(\hat{\mathcal{U}}) \left\{\begin{array}{ll}
\max & N(g)\\
\mathrm{s.t. } \ & g_k^i \in \partial \relu(\theta_k^i(x)), \    \forall i\in\{1,\dots, n_k\}, k\in \{1,\cdots,L-1\}\\
& g_k \in \{0,1\}^{n_k}, \forall k\in \{1,\cdots,L-1\}\\
& x\in \xcal.
\end{array}\right.
$$
\label{prop:interv to set}
\end{proposition}

\begin{proof}
Let $g^*$ denote a solution of~$(\mathcal{U}^*)$ and $\hat{g}$ a solution of~$(\hat{\mathcal{U}})$.

The vector $\hat{g}$ is naturally feasible for~$(\mathcal{U}^*)$ and thus 
$$N(\hat{g}) \leq N(g^*).$$

Now, if $(g^*)_k^i \in ]0,1[$ for some $i$ and $k$, using convexity of $g_k^i \mapsto N(g)$, there exists $\overline{g},$ whose all components are equal to that of $g^*$ except $(\overline{g})_k^i$ which belongs to $ \{0,1\}$, such that $N(\overline{g}) \geq N(g^*)$. Moreover $(g^*)_k^i \in ]0,1[ \cap \partial \relu(\theta_k^i(x))$ implies that $\theta_k^i(x) = 0$ and thus $(\overline{g})_k^i \in \{0,1\}$ is also in $\partial \relu(\theta_k^i(x))$. Repeating this for all components of $g^*$ that are in $]0,1[$, we show that there exists a solution $\overline{g}$ of~$(\mathcal{U}^*)$ whose components are all in $\{0,1\}$  satisfying $N(g^*) \leq N(\overline{g})$. Therefore, we have 
$$N(g^*) \leq N(\overline{g})\leq  N(\hat{g}).$$

\end{proof}

Using Proposition~\ref{prop:interv to set} and the fact that $g_k^i\in \partial \relu(\theta_k^i(x))$ is equivalent to $\left(g_k^i-\frac{1}{2}\right) \theta_k^i(x)\geq 0$ for $g_k^i \in \{0,1\}$, we compute an upper bound of the Lipschitz constant by solving 

$$(\overline{\mathcal{P}}) \left\{
\begin{array}{ll}
\max & N(g)\\
\mathrm{s.t. } \ & (g_k^i-\frac{1}{2}) \theta_k^i(x)\geq 0, \    \forall i\in\{1,\dots, n_k\}, k\in \{1,\cdots,L-1\}\\
& g_k \in \{0,1\}^{n_k}, \forall k\in \{1,\cdots,L-1\}\\
& x\in \xcal.
\end{array}
\right.
$$

By the definition of the activation region, problem ($\mathcal{L}^*$) can be reformulated as
$$(\underline{\mathcal{P}})\left\{
\begin{array}{ll}
\max & N(\sigma)\\
\mathrm{s.t. } \ & (\sigma_k^i-\frac{1}{2}) \theta_k^i(x)>0, \   \forall i\in\{1,\dots, n_k\}, k\in \{1,\cdots,L-1\}\\
& \sigma_k \in \{0,1\}^{n_k},  \forall k\in \{1,\cdots,L-1\}\\
& x\in \xcal.
\end{array}
\right.
$$
In order to avoid strict inequalities, we introduce for $\varepsilon\geq 0$, 
$$(\underline{\mathcal{P}})_\varepsilon \left\{
\begin{array}{ll}
\max & N(\sigma)\\
\mathrm{s.t. } \ & \left(\sigma_k^i-\frac{1}{2}\right) \theta_k^i(x)\geq \varepsilon, \    \forall i\in\{1,\dots, n_k\}, k\in \{1,\cdots,L-1\}\\
& \sigma_k \in \{0,1\}^{n_k}, \forall k\in \{1,\cdots,L-1\}\\
& x\in \xcal.
\end{array}
\right.
$$

Observe that $(\underline{\mathcal{P}})_0$ corrresponds to $(\overline{\mathcal{P}})$.\\

We now introduce the following constraint sets: 
\begin{eqnarray}
\Ccal
& = 
&\xcal \times \{0,1\}^{n_1+n_2+...+n_{L-1}},
\\
\Ccal_s
& = 
& \left\{
(x,\sigma)\in\Ccal : \left(\sigma_k^i-\frac{1}{2}\right) \theta_k^i(x) > 0, \    \forall i\in\{1,\dots, n_k\}, k\in \{1,\cdots,L-1\}
\right\},
\\
\Ccal_\varepsilon
& = 
& \left\{
(x,\sigma)\in\Ccal : \left(\sigma_k^i-\frac{1}{2}\right) \theta_k^i(x)\geq \varepsilon, \    \forall i\in\{1,\dots, n_k\}, k\in \{1,\cdots,L-1\}
\right\},
\end{eqnarray}
so that the preceeding problems can be rewritten as:
$$
(\overline{\mathcal{P}}) \left\{
\begin{array}{ll}
\max_{x,\sigma}
& N(\sigma)
\\
\mathrm{s.t. } 
& (x,\sigma)\in\Ccal_0
\end{array}
\right\}
\qquad
(\underline{\mathcal{P}}) \left\{
\begin{array}{ll}
\max_{x,\sigma}
& N(\sigma)
\\
\mathrm{s.t. } 
& (x,\sigma)\in\Ccal_s
\end{array}
\right\}
\qquad
(\underline{\mathcal{P}})_\varepsilon \left\{
\begin{array}{ll}
\max_{x,\sigma}
& N(\sigma)
\\
\mathrm{s.t. } 
& (x,\sigma)\in\Ccal_\varepsilon
\end{array}
\right\}.
$$

Let $\overline{L}$, $\underline{L}$, $\underline{L}_{\varepsilon}$ denote the optimal values of  $(\overline{\mathcal{P}})$, $(\underline{\mathcal{P}})$ and $(\underline{\mathcal{P}})_\varepsilon$ respectively.
The following proposition summarizes and completes the above discussion.

\begin{proposition}
We have
\begin{enumerate}
\item The function $]0,+\infty[\ni \varepsilon \mapsto \underline{L}_\varepsilon$  is non-increasing and piece-wise constant.
\item $\lim_{\varepsilon \downarrow 0}\underline{L}_{\varepsilon} \uparrow \underline{L}\leq L(f,\xcal)\leq \overline{L} =  \underline{L}_0 .$ 
\item If $\cup_{i,k}H_{i,k}$ is  Lebesgue measure negligible, then $L(f,\xcal)=\underline{L}$.
\item If the \relu network $f$ is in general position (see Definition 4 in \cite{Matt2020}), then  $L(f,\xcal)=\overline{L}=\underline{L}$.
\end{enumerate}
\end{proposition}

\begin{proof}
\begin{enumerate}
    \item If $\varepsilon_1\leq \varepsilon_2$ then $\Ccal_{\varepsilon_2}\subset \Ccal_{\varepsilon_1}$ which implies that $\underline{L}_\varepsilon$ is non increasing. Moreover, $\underline{L}_\varepsilon$ belongs to $\{N(\sigma), \sigma \in \{0,1\}^{n_1+n_2+...+n_{L-1}}\}$ which is a finite set. Hence $\underline{L}_\varepsilon$ is piece-wise constant.
    \item Observe that $\underline{L}_{\varepsilon} \leq \underline{L}\leq L(f,\xcal)\leq \overline{L} =  \underline{L}_0.$ To prove the remaining statement, let $\hat{\sigma}$ an optimal solution of $(\underline{\mathcal{P}})$. Then $\mathcal{R}(\xcal;\hat{\sigma})\neq \emptyset$, that is there exists $\hat{x}\in\xcal$ such that $\left(\hat{\sigma}_k^i-\frac{1}{2}\right) \theta_k^i(\hat{x})>0, \ \forall i\in\{1,\dots, n_k\},  \forall k\in \{1,\cdots,L-1\}.$ Define 
    $$\hat{\varepsilon}=\min_{i\in\{1,\dots, n_k\}, k\in \{1,\cdots,L-1\}}\left(\hat{\sigma}_k^i-\frac{1}{2}\right) \theta_k^i(\hat{x}).$$ 
    So $(\hat{x},\hat{\sigma}) \in \Ccal_\varepsilon$ for all $\varepsilon\leq \hat{\varepsilon}$ and thus $\underline{L}_\varepsilon \geq \underline{L}.$
    
    \item We can write $\xcal = (\bigcup_{\sigma} \mathcal{R}(\xcal;\sigma))\cup (\bigcup_{i,k} H_{i,k}\cap \xcal).$ Now (\ref{eq:ess-sup}) implies
    \begin{eqnarray*}
    L(f,\xcal)&=& \esup_{x\in  (\bigcup_{\sigma} \mathcal{R}(\xcal;\sigma))}\sup_{G\in \mathcal{J}f(x)} \|G\|  \\
     &=& \sup_{\sigma|\mathcal{R}(\xcal,\sigma)\neq \emptyset} N(\sigma)\\
     &=& \underline{L}.
    \end{eqnarray*}
    \item By \cite[Theorem 2]{Matt2020} if the \relu network is in general position then 
    $$\mathcal{J}f(x)= \left\{\|M_L\diag(g_{L-1}) M_{L-1}\cdots \diag(g_1) M_1\| \quad | \quad   g_k \in [0,1]^{n_k},\  g_k^i\in \partial \relu(\theta_k^i(x)) \right\}.$$   By (\ref{eq:ess-sup}) we obtain $L(f,\xcal)=\overline{L}.$ Furthermore, if the \relu network $f$ is in a general position then $\cup_{i,k}H_{i,k}$ is Lebesgue measure negligible \cite{Matt2020}, which ensures the second equality using the preceding item.
\end{enumerate}

\end{proof}

\begin{remark}
  By \cite[Theorem 3]{Matt2020} the set of \relu networks not in general position has Lebesgue measure zero over the
parameter space and consequently for almost all \relu  Networks we have $L(f,\xcal)=\overline{L}=\underline{L}$.
\end{remark}

\begin{corollary}
 If $\xcal =\mathbb{R}^{n_0}$ and the bias $b_k$, $k\in\{1,\cdots, L-1\}$ are zero, then $\underline{L}_\varepsilon=\underline{L}$ for all $\varepsilon>0$.
\end{corollary}

\begin{proof}

We have already established that $\underline{L}_\varepsilon\leq \underline{L}.$ Let $(\hat{x},\hat{\sigma})\in \Ccal_s$ an optimal solution of $(\underline{\mathcal{P}})$ and let $\varepsilon>0.$ Since the bias  $b_k$, $k\in\{1,\cdots, L-1\}$ are zero then for any $\lambda>0$ we have $\theta_k(\lambda \hat{x})=\lambda \theta_k(\hat{x}).$ Therefore we can choose $\lambda$ sufficiently large so that 
$
\left(\hat{\sigma}_k^i-\frac{1}{2}\right) \theta_k^i(\lambda\hat{x}) > \varepsilon, \    \forall i\in\{1,\dots, n_k\}, k\in \{1,\cdots,L-1\}
$
ensuring that $(\lambda \hat{x},\hat{\sigma}) \in  \Ccal_\varepsilon$ and consequently 
$\underline{L}_\varepsilon\geq N(\hat{\sigma})= \underline{L}.$
\end{proof}
The following simple examples illustrate the above results. 
\begin{example}
For $f(x)= x-1$, we trivially have $L(f,\xcal) = 1$. This function can also be written $f(x) = \max(x-1,0)-\max(1-x,0)$,  corresponding to our formalism with $$M_1=\begin{pmatrix}
1\\
-1
\end{pmatrix},  \ M_2=\begin{pmatrix}
1 & -1
\end{pmatrix},\ b_1=\begin{pmatrix}
-1\\
1
\end{pmatrix}, \ b_2=0.
$$
With this choice, we have $\underline{L}_\varepsilon=\underline{L}=1$ for all $\varepsilon>0$ 
 and 
$\underline{L}_0=\overline{L}=2$.  
Indeed, $x = 1$ is neither in $\mathcal{R}(\mathbb{R};\sigma)$ nor in the feasible set of $(\underline{\mathcal{P}})_\varepsilon$ for $\varepsilon > 0$, but it belongs to $\Ccal_0$.

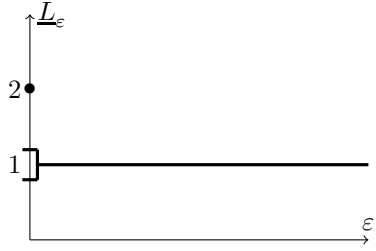
\centerline{
\begin{tikzpicture}[xscale=1]
\clip (-0.5,-0.5) rectangle(5,3.5) ;

\draw[black,->] (0,0)--(4.5,0) ;
\draw (4.5,0) node[above] {$\varepsilon$} ;

\draw[black,->] (0,0)--(0,3) ;
\draw (0.3,3) node {$\underline{L}_\varepsilon$} ;

\draw (-0.2,2) node{$2$} ;
\draw[black] (0,2) node{$\bullet$} ;

\draw (-0.2,1) node {$1$} ;
\draw[black,very thick] (0.1,0.8) -- (0.1,1.2) ;
\draw[black,very thick] (-0.1,0.8) -- (0.1,0.8) ;
\draw[black,very thick] (-0.1,1.2) -- (0.1,1.2) ;
\draw[black,very thick] (0.1,1)--(4.5,1) ;
\end{tikzpicture}
}

\end{example}
\begin{example}
For $f(x)=\max(x+1,0)-\max(x-1,0)$ corresponding to $$M_1=\begin{pmatrix}
1\\
1
\end{pmatrix},  \ M_2=\begin{pmatrix}
-1 & 1
\end{pmatrix},\ b_1=\begin{pmatrix}
-1\\
1
\end{pmatrix}, \ b_2=0,
$$
we have $\underline{L}=1$, 
$L(f,\xcal) = 1$ and 
$\underline{L}_0=\overline{L}=1$. 
Moreover, 
$$
\underline{L}_\varepsilon= \left\{\begin{array}{ll}
1 &  \text{ for all } 0<\varepsilon\leq 0.5,
\\
0 & \text{ for all } \varepsilon >0.5.
\end{array}\right.
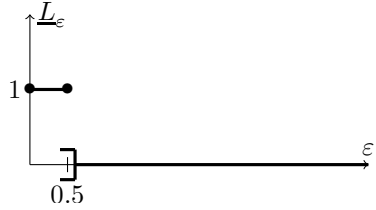
$$

\centerline{
\begin{tikzpicture}[xscale=1]
\clip (-0.5,-0.5) rectangle(5,2.5) ;

\draw[black,->] (0,0)--(4.5,0) ;
\draw (4.5,0) node[above] {$\varepsilon$} ;

\draw[black,->] (0,0)--(0,2) ;
\draw (0.3,2) node {$\underline{L}_\varepsilon$} ;

\draw (-0.2,1) node{$1$} ;
\draw[black] (0,1) node{$\bullet$} ;
\draw[black,very thick] (0.0,1.0) -- (0.5,1) ;
\draw[black] (0.5,1) node{$\bullet$} ;

\draw[black,very thick] (0.6,-0.2) -- (0.6,0.2) ;
\draw[black,very thick] (0.4,-0.2) -- (0.6,-0.2) ;
\draw[black,very thick] (0.4,0.2) -- (0.6,0.2) ;
\draw[black,very thick] (0.6,0)--(4.5,0) ;

\draw[black] (0.5,-0.1)--(0.5,0.1) ;
\draw (0.5,-0.4) node {$0.5$} ;
\end{tikzpicture}
}
\end{example}

\section{MIQCQP reformulations}\label{MIQCQP}


Let now consider the $L^p$-norm for $p\in\{1,2,\infty\}$ and assume that $\xcal$  can be expressed by quadratic constraints (e.g. a ball) or linear constraints (e.g. polyhedron).

Then let us remark that 
we have $\rho_k \circ \theta_k(x)= \sigma_{k}\odot \theta_{k}(x)$ for all $k=1,2,\cdots {L-1}$ as soon as $(x, \sigma)$ is in one of the sets $\Ccal_\varepsilon$, $\Ccal_s$ or $\Ccal_0$. By definition of $\theta_k(x)$, we have  thus 
$\rho_k \circ \theta_k(x)= \sigma_{k}\odot \theta_{k}(x)= \sigma_k \odot (M_k\circ\rho_{k-1}\circ \theta_{k-1}(x) +b_{k}) $ for $k=2,3,\cdots {L-1}$ and  $\rho_1 \circ \theta_1(x)= \sigma_{1} \odot (M_1 x +b_1).$  Denoting $x_k=\rho_k \circ \theta_k(x)$ we get the following bilinear relations
$$
x_k= \sigma_k\odot (M_kx_{k-1}+b_{k}) \mbox{ for } k=1,2,\cdots {L-1} \mbox{ with } x_0= x.
$$
Therefore, the constraints of $(\underline{\mathcal{P}})_\varepsilon$ are equivalent to
$$
x_0\in \xcal \text{ and for all } k\in \{1,2,\dots, {L-1}\}, 
\left\{\begin{array}{l}
x_k= \sigma_k\odot (M_kx_{k-1}+b_{k}),
\\
(\sigma_k-\frac{1}{2}) \odot (M_k x_{k-1} +b_k) \geq \varepsilon,
\\ 
\sigma_k \in \{0,1\}^{n_k}.
\end{array}\right.
$$

The objective function of $(\underline{\mathcal{P}})_\varepsilon$ can also be expressed, depending on the chosen $L^p$-norm, as
\begin{equation*}\label{norme-def}
N_p(\sigma) = \|M_L\diag(\sigma_{L-1}) M_{L-1}\cdots \diag(\sigma_1) M_1\|_p=\sup_{y, \|y\|_p\leq 1} \|M_L\diag(\sigma_{L-1}) M_{L-1}\cdots \diag(\sigma_1) M_1y\|_p
\end{equation*}
 and equivalently as
 $$
 \begin{array}{ll}
 \displaystyle\max_{y} & \|y_{L}\|_p\\
 \mathrm{s.t. }    & y_k= M_k \diag(\sigma_{k-1}) y_{k-1}, \forall k\in \{2,\dots, L\}, \\
                   & y_1 =M_1 y_0\\
                   & \|y_0\|_p\leq 1
 \end{array}
 $$
where $y$ is the collection of $y_0, y_1, \dots, y_L$.
 

Finally, for a given $p\in\{1,2,\infty\}$, $(\underline{\mathcal{P}})_\varepsilon$ is equivalent to the following problem:
 $$(\underline{\mathcal{P}}_p)_\varepsilon 
 \left\{
\begin{array}{ll}
\displaystyle\max_{y, \sigma, x} & \|y_{L}\|_p\\
\mathrm{s.t. }\ & y_k= M_k \diag(\sigma_{k-1}) y_{k-1}, \forall k\in \{2,\dots, L \}, \\
                 & y_1 =M_1 y_0\\
                   & x_k= \sigma_k\odot (M_kx_{k-1}+b_{k}),  k=1,2,\cdots {L-1}\\
                   & (\sigma_k-\frac{1}{2}) \odot (M_k x_{k-1} +b_k) \geq \varepsilon, \   k=1,2,\cdots {L-1}\\
                   & \sigma_k \in \{0,1\}^{n_k}, k=1,2,\cdots {L-1}\\
                   & x_0\in \xcal\\
                   & \|y_0\|_p\leq 1
\end{array}
\right.
$$
where $\sigma$ is the collection of $\sigma_1, \dots, \sigma_{L-1}$ and $x$ is the collection of $x_0, x_1, \dots, x_{L-1}$. Provided that $\xcal$ is expressible by quadratic (or linear) constraints, the above problem can also be expressed as a MIQCQP as shown below.

\paragraph{The case $p=2$}
In this case, $(\underline{\mathcal{P}}_2)_\varepsilon$ is trivially equivalent to the following MIQCQP problem.
 $$
\begin{array}{ll}
\displaystyle\max_{y, \sigma, x} & \|y_{L}\|_2^2\\
\mathrm{s.t. }\ & y_k= M_k \diag(\sigma_{k-1}) y_{k-1}, \forall k\in \{2,\dots, L \}, \\
                   & y_1 =M_1 y_0\\
                   & x_k= \sigma_k\odot (M_kx_{k-1}+b_{k}),  k=1,2,\cdots {L-1}\\
                   & (\sigma_k-\frac{1}{2}) \odot (M_k x_{k-1} +b_k) \geq \varepsilon, \   k=1,2,\cdots {L-1}\\
                   & \sigma_k \in \{0,1\}^{n_k}, k=1,2,\cdots {L-1}\\
                   & x_0\in \xcal\\
                   & \|y_0\|_2^2\leq 1.\\
\end{array}
$$
\paragraph{The case $p=1$}
Observing that for all $a\in\mathbb{R}$, $|a|=(2\lambda-1)a$ with $\lambda\in\{0,1\}$ and $(2\lambda-1)a\geq0$, $(\underline{\mathcal{P}}_1)_\varepsilon$ is equivalent to
 $$
\begin{array}{ll}
\displaystyle\max_{y, \sigma, x,\nu,\mu} &\sum_{i=1}^{n_L} (2\mu_i-1)(y_L)_i\\
\mathrm{s.t. }\ & y_k= M_k \diag(\sigma_{k-1}) y_{k-1}, \forall k\in \{2,\dots, L \}, \\
                   & y_1 =M_1 y_0\\
                   & x_k= \sigma_k\odot (M_kx_{k-1}+b_{k}),  k=1,2,\cdots {L-1}\\
                   & (\sigma_k-\frac{1}{2}) \odot (M_k x_{k-1} +b_k) \geq \varepsilon, \   k=1,2,\cdots {L-1}\\
                   & \sigma_k \in \{0,1\}^{n_k}, k=1,2,\cdots {L-1}\\
                   & x_0\in \xcal\\[0.2cm]
                   & \nu\in \{0,1\}^{n_0}\\
                   &(2 \nu_i-1) (y_0)_i \geq 0,\ i=1,2,\cdots,n_0\\
                   &\sum_{i=1}^{n_0} (2\nu_i-1)(y_0)_i\leq 1, i=1,2,\cdots,n_0\\[0.2cm]
                   & \mu \in \{0,1\}^{n_L}\\
                   &(2 \mu_i-1) (y_L)_i \geq 0,\ i=1,2,\cdots,n_L.
\end{array}
$$
The last constraints can be further linearized by introducing additional binary variables and using a \textit{big-M} technique. Indeed, for all $x\in \mathbb{R}$, with $|x|\leq B$, $u=|x|$  if and only if there exists $\lambda \in \{0,1\}$ such that 
 $$
 \left\{
\begin{array}{l}
u\geq x \mbox{ and } u\geq -x,\\
x \leq B(1-\lambda)\mbox{ and } x \geq - B\lambda,\\
u \leq -x +  2B(1-\lambda) \mbox{ and } u\leq x+ 2B\lambda.
\end{array}
\right.
$$
The first constraints imply that $u\geq |x|.$ The second ones ensure that if $x>0$ (respectively $x<0$)  then $\lambda=0$ (respectively $\lambda=1$). The last constraints guarantee that $u\leq |x|.$  Problem $(\underline{\mathcal{P}}_1)_\varepsilon$ is hence equivalent to

$$
\begin{array}{ll}
\displaystyle\max_{y, \sigma, x,u,w,\nu,\mu} &\sum_{i=1}^{n_L} w_i\\
\mathrm{s.t. }\ & y_k= M_k \diag(\sigma_{k-1}) y_{k-1}, \forall k\in \{2,\dots, L \}, \\
                   & y_1 =M_1 y_0\\
                   & x_k= \sigma_k\odot (M_k x_{k-1}+b_{k}),  k=1,2,\cdots {L-1}\\
                   & (\sigma_k-\frac{1}{2}) \odot (M_k x_{k-1} +b_k) \geq \varepsilon, \   k=1,2,\cdots {L-1}\\
                   & \sigma_k \in \{0,1\}^{n_k}, k=1,2,\cdots {L-1}\\
                   & x_0\in \xcal\\[0.2cm]
                   & \nu\in \{0,1\}^{n_0}\\
                   & -1\leq y_0 \leq 1\\
                   & (y_0)_i\leq u_i,\ i =1,2,\cdots,n_0\\
                   & -(y_0)_i\leq u_i, \ i=1,2,\cdots,n_0\\
                   & (y_0)_i \leq 1-\nu_i, i =1,2,\cdots,n_0\\
                   & (y_0)_i\geq -\nu_i, \ i =1,2,\cdots,n_0\\
                   & u_i\leq -(y_0)_i +  2(1-\nu_i), \ i=1,2,\cdots,n_0\\
                   & u_i\leq (y_0)_i +  2\nu_i, i \  =1,2,\cdots,n_0\\
                   & \sum_{i=1}^{n_0} u_i\leq 1\\[0.2cm]
                   & \mu \in \{0,1\}^{n_L}\\
                   & (y_L)_i\leq w_i,\ i =1,2,\cdots,n_L\\
                   & -(y_L)_i\leq w_i, \ i=1,2,\cdots,n_L\\
                   & (y_L)_i\leq C(1-\mu_i), i =1,2,\cdots,n_L\\
                   & (y_L)_i\geq -C\mu_i, \ i =1,2,\cdots,n_L\\
                   & w_i\leq -(y_L)_i + 2C (1-\mu_i), \ i=1,2,\cdots,n_L\\
                   & w_i\leq (y_L)_i + 2C \mu_i, i =1,2,\cdots,n_L\\
\end{array}
$$
where $C\gg0$ is the so-called \textit{big-M} constant.

\paragraph{The case $p=\infty$}
Problem $(\underline{\mathcal{P}}_\infty)_\varepsilon$ can similarly be expressed as
 $$
\begin{array}{ll}
\displaystyle\max_{y, \sigma, x,u,\mu,\eta} &\sum_{i=1}^{n_L} \eta_i u_i \\
\mathrm{s.t. }\ & y_k= M_k \diag(\sigma_{k-1}) y_{k-1}, \forall k\in \{2,\dots, L \}, \\
                   & y_1 =M_1 y_0\\
                   & x_k= \sigma_k\odot (M_kx_{k-1}+b_{k}),  k=1,2,\cdots {L-1}\\
                   & (\sigma_k-\frac{1}{2}) \odot (M_k x_{k-1} +b_k) \geq \varepsilon, \   k=1,2,\cdots {L-1}\\
                   & \sigma_k \in \{0,1\}^{n_k}, k=1,2,\cdots {L-1}\\
                   & x_0\in \xcal\\
                   & -1\leq y_0\leq 1\\[0.2cm]               
                   & \mu \in \{0,1\}^{n_L}\\
                   & u\geq 0\\
                   & u_i=(2\mu_i-1)(y_L)_i,\ i=1,2,\cdots,n_L\\
                   & \eta \in \{0,1\}^{n_L}\\
                   & \sum_{i=1}^{n_L} \eta_i=1.
\end{array}
$$
Note: in this formulation the last constraints can also be linearized as in the case $p=1$.

\bibliographystyle{plain}
\bibliography{references}
\end{document}